\documentclass[11pt]{amsart}
\usepackage{eurosym}
\usepackage[pagebackref,colorlinks=true]{hyperref}
\usepackage{verbatim}
\usepackage{eucal,url,amssymb,stmaryrd,enumerate,amscd,}
\usepackage{amsfonts}
\usepackage{amsmath,amsthm,amssymb,amscd,enumerate,eucal,url,stmaryrd}
\usepackage[margin=1in]{geometry}

\numberwithin{equation}{section}
\newtheorem{thrm}{Theorem}[section]
\newtheorem{lemma}[thrm]{Lemma}
\newtheorem{prop}[thrm]{Proposition}

\newtheorem{dfn}[thrm]{Definition}
\newtheorem{rmrk}[thrm]{Remark}
\newtheorem{exm}[thrm]{Example}

\newtheorem{conv}[thrm]{Convention}

  \allowdisplaybreaks

\title[Applications of Weak Metric Structures to Non-Symmetrical Gravitational Theory]
{Applications of Weak Metric Structures \\ to Non-Symmetrical Gravitational Theory
}
\date{\today}

\author{Milan Zlatanovi\'c}
\address
{Milan Zlatanovi\'c: Department of Mathematics, Faculty of Science and Mathematics, University of Ni\v s, Vi\v segradska 33, 18000 Ni\v s, Serbia}
\email{zlatmilan@yahoo.com}

\author{Vladimir Rovenski}
\address
{Vladimir Rovenski: Department of Mathematics,
Faculty of Natural Science, University of Haifa,
Mount Carmel,
3498838 Haifa, Israel}
\email{vrovenski@univ.haifa.ac.il}

\begin{document}

\begin{abstract}
Linear connections satisfying the Einstein metricity condition
are important in the study of generalized Riemannian manifolds
$(M,G=g+F)$, where the symmetric part $g$ of $G$ is a non-degenerate $(0,2)$-tensor, and $F$ is the skew-symmetric part.
Such structures naturally arise in spacetime models in theoretical physics, where $F$ can be defined as an almost complex or almost contact metric (a.c.m.) structure. In the paper, we first study more general models, where $F$ has constant rank and is based on weak metric structures (introduced by the second author and R.~Wolak), which generalize almost complex and a.c.m. structures. We consider linear connections with totally skew-symmetric torsion that satisfy both the Einstein metricity condition and the $A$-torsion condition, where $A$ is a skew-symmetric (1,1)-tensor adjoint to~$F$. In the almost Hermitian case, we prove that the manifold with such a connection is weak nearly K\" ahler, the torsion is completely determined by the exterior derivative of the fundamental 2-form and the Nijenhuis tensor, and the structure tensors are parallel, while in the weak a.c.m. case, the contact distribution is involutive, the Reeb vector field is Levi-Civita parallel, and the structure tensors are also parallel with respect to both connections.
For rank$(F)=\dim M$, we apply weak almost Hermitian structures to fundamental results (by the first author and S. Ivanov) on
generalized Riemannian manifolds and prove that the manifold
equipped with an Einstein's connection is a weighted product of several nearly K\"ahler manifolds.
For~rank$(F)<\dim M$ we apply weak almost Hermitian and weak a.c.m. structures and obtain splitting results for generalized Riemannian manifolds equipped with Einstein's connections.


\end{abstract}

\date{\today }
\maketitle


\setcounter{tocdepth}{2}

\section{Introduction}

\subsection{Motivation From General Relativity}

General relativity (GR) was developed by A. Einstein in 1916
\cite{Ein1}, with contributions by many others after 1916.
In GR the equation $ds^2=g_{ij}dx^idx^j$
is valid, where $g_{ij}=g_{ji}$ are functions of the point in the space. In GR which is the four dimensional space-time continuum metric properties depend on the mass distribution.
The magnitudes $g_{ij}$  are known as \emph{gravitational
potential}.
The Christoffel symbols, usually denoted by
$\Gamma^k_{ij}$, play the role of magnitudes defining the gravitational force field. General relativity explains gravity as a curvature of spacetime.
In the GR a metric tensor is related by the Einstein equations
\begin{equation}\label{E-GR1}
 R_{ij}-\frac 12\,R\,g_{ij}=T_{ij},
\end{equation}
where $R_{ij}$ is Ricci tensor of metric of space time, $R$ is
scalar curvature of metric, and $T_{ij}$ is energy-momentum tensor of matter.
However, in 1922, A. Einstein believed that the universe is apparently static, and since a static cosmology was not supported by the GR field equations, he added a cosmological constant $\Lambda$ to the field equations \eqref{E-GR1}, which became
\begin{equation}\label{E-GR2}
 R_{ij}-\frac 12\,R\,g_{ij}+\Lambda\,g_{ij} =T_{ij}. \end{equation}
Since 1923 until the end of his life, A. Einstein worked on various versions of the Unified Field Theory (Non-symmetric Gravitational Theory -- NGT) \cite{Ein}. This theory was intended to unite the gravitation theory, to which GR is related, and the theory of electromagnetism.
Introducing various versions of his NGT, A. Einstein used a complex basic tensor, with a symmetric real part and a skew-symmetric imaginary part. Beginning in 1950, A. Einstein used a real non-symmetric basic tensor $G$, sometimes called a \emph{generalized Riemannian metric}/\emph{manifold}).
Note that in the NGT the symmetric part $g_{ij}$ of the basic
tensor $G_{ij} \ (G_{ij}=g_{ij}+F_{ij})$ is associated with gravity, and the skew-symmetric part $F_{ij}$ is associated with electromagnetism.
The same is true for the symmetric part of the connection and torsion tensor, respectively.

Later, ideas of non-symmetric metric tensor appeared in Moffat's non-symmetric theory of gravity \cite{Mof}.
Moffat's theory extends general relativity by introducing a non-symmetric metric and connection, allowing for richer gravitational dynamics. In~his theory the antisymmetric part is a Proca field (massive Maxwell's field), which is part of the gravitational interaction that contributes to the rotation of~galaxies. The~connection in NGT exhibits regularity, allowing for smooth and physically meaningful solutions without singularities.  

In NGT \cite{EGY}, two new classes of path equations are derived using Baza\'{n}ski's variational approach. These equations describe how test particles and charged particles move, and they also reflect certain quantum characteristics of NGT.
An explicit formula of such a connection, satisfying the \textit{Einstein metricity condition} (EMC), is obtained by localizing the global formula given recently by S. Ivanov and M.~Zlatanovi\' c \cite{IZ1}. These equations not only reveal quantum features intrinsic to the NGT, but also highlight possible interactions between torsion and electromagnetic potential, even in the absence of electromagnetic force.
In \cite{EGY2}, the authors establish new identities for a connection with totally skew-symmetric torsion on NGT. These identities, derived via the Dolan-McCrea variational method, naturally split into symmetric and skew-symmetric parts, generalizing the second Bianchi identity.
%
Recent approaches to modified gravity often rely on post-Riemannian geometry, incorporating torsion and non-metricity as natural extensions of Einstein's GR framework~\cite{IOS}. Such formulations resonate with the original ideas of Einstein's NGT.

While in a Riemannian space the connection coefficients are
expressed through the metric, $g_{ij}$, in Einstein's works on NGT the connection between these magnitudes is determined by the EMC,
i.e. the non-symmetric metric tensor $G$ and the connection components $\Gamma_{ij}^k$ are related by the equations
\begin{equation}\label{metein-0}
 \frac{\partial G_{ij}}{\partial x^m}-\Gamma ^p_{im}G_{pj}
-\Gamma^p_{mj}G_{ip}=0.
\end{equation}
A generalized Riemannian manifold satisfying the EMC \eqref{metein-0} is also called an NGT-space \cite{Ein,Mof}.
The~choice of a connection in NGT is not uniquely determined. In particular, in NGT there exist two types of covariant derivatives, for example for the tensor $a^i_j$:
\begin{equation*}
 a{}^{\underset+i}_{\underset+j|m}=\frac{\partial a^i_{j}}{\partial x^m}
 +\Gamma^i_{pm}a^p_j-\Gamma^p_{jm}a^i_p; \quad
 a{}^{\underset-i}_{\underset-j|m}=\frac{\partial a^i_{j}}{\partial
 x^m}+\Gamma^i_{mp}a^p_j-\Gamma^p_{mj}a^i_p,
 \end{equation*}
where the lowering and the rising of indices is defined by
equations
\begin{equation}\label{e2}
G_{pi}G^{pj}=G_{ip}G^{jp}=\delta^j_i.
\end{equation}
A. Einstein considered in NGT only one curvature tensor:
\begin{equation}\label{e}
R^i_{klm}=\Gamma^i_{kl,m}-\Gamma^i_{km,l}-\Gamma^i_{sl}\Gamma^s_{km}+\Gamma^i_{sm}\Gamma^s_{kl},
 \end{equation}
and proved a Bianchi type identity for the covariant curvature tensor $R{}_{iklm}=G_{si}R^{s}{}_{klm}$ (see \cite{Ein1950}):
$$
 R{}_{\underset{-+-+}{iklm}|\,n}+R{}_{\underset{-+++}{ikmn}|\,l} + R{}_{\underset{-+--}{iknl}|\,m}=0.
$$
\subsection{Main Objectives and Structure of the Paper}

The main goal of the paper is to study a generalized Riemannian manifold $(M,G=g+F)$ equipped with a linear connection satisfying EMC (\ref{metein-0}) and admitting a totally skew-symmetric torsion. 
Our key results show that the assumption of the $A$-torsion condition, see \eqref{Eq-A-T}, or equivalently, the preservation of $g$, leads to a simplified form of the Nijenhuis tensor and characterizes weak nearly K\" ahler or weak nearly cosymplectic structures.
In Section~\ref{sec:geomod}, we represent the geometric model with the EMC, define the $A$-torsion condition and prove Proposition~\ref{Pr-2.7}.
%
In Section~\ref{sec:03}, we study application of weak metric structures on generalized Riemannian manifolds, focusing on the interplay between a skew-symmetric endomorphism \(A\) and a self-adjoint endomorphism \(Q>0\). Under the key assumption that \(A\) commutes with \(Q\), \([A, Q] = 0\), we establish conditions under which these tensors admit block-diagonalization, and analyze the existence of linear connections with totally skew-symmetric torsion compatible with the given structures. In~particular, we prove that weak almost Hermitian manifolds admitting Einstein's connections are weak nearly K\"ahler, and describe their decomposition into weighted products of nearly K\"ahler manifolds. Theorems~\ref{mainweak} and \ref{Th-001} show that if an Einstein's connection $\nabla$
satisfies 
the $A$-torsion condition on a weak almost Hermitian manifold, then \(Q\) is parallel with respect to both connections, $\nabla$ and the Levi-Civita connection, and the manifold is weak nearly K\"ahler.

For a weak almost contact metric (a.c.m.) manifold equipped with 
an Einstein's connection $\nabla$ satisfying 
the \(A\)-torsion condition, we have proved that the Reeb vector field \(\xi\) is parallel with respect to the Levi-Civita connection, the Nijenhuis tensor \(N^{wac}\) is totally skew-symmetric, and the contact distribution \(\mathcal{D} = \ker \eta\) is involutive,
the tensor \(Q\) is parallel with respect to $\nabla$
and the Levi-Civita and connection. 
Based on Proposition~\ref{L-3.16} on almost-nearly cosymplectic manifolds, we present Theorem~\ref{T-3.8IZ}, which complements \cite[Theorem~3.8]{IZ1}.
Theorems~\ref{T-3.18} and \ref{Th-3.13} show that if \(Q\) is conformal when restricted to \(\mathcal{D}\), then the manifold locally splits as a weighted product of \(\mathbb{R}\) and a nearly K\" ahler manifold; otherwise, \(\mathcal{D}\) decomposes into mutually orthogonal eigen-distributions of $Q$ with constant eigenvalues, locally giving a weighted product structure of \(\mathbb{R}\) and several nearly K\" ahler manifolds.

\section{Geometric Model}
\label{sec:geomod}

The fundamental (0,2)-tensor $G$ in a non-symmetric (generalized) Riemannian manifold $(M,G)$ is, in general, non-symmetric. It decomposes in two parts, $G=g+F$,
the symmetric part $g$ (called Riemannian metric) and the
skew-symmetric part $F$ (called fundamental 2-form), where
\begin{equation}\label{metric}
g(X,Y)=\frac12\big[G(X,Y)+G(Y,X)\big], \qquad F(X,Y)=\frac12\big[G(X,Y)-G(Y,X)\big].
 \end{equation}
We assume that the symmetric part, $g$, is non-degenerate of arbitrary signature, and the skew-symmetric part, $F\ne0$, has a constant rank, e.g., is non-degenerate.
Therefore, we obtain a well-defined (1,1)-tensor $A\ne0$ of constant rank determined by the following condition:
\begin{equation}\label{m1}
 g(AX,Y) = F(X,Y)\quad \text{for all } X, Y \in \mathfrak{X}_M.
\end{equation}
According to the above, since $F$ is skew-symmetric, the tensor $A$ is also skew-symmetric:
\[
 g(AX, Y) = -g(X, AY) \quad \text{for all } X, Y \in \mathfrak{X}_M.
\]

\begin{conv}
\label{Conv-1}\rm
In the whole paper we shall use the capital Latin letters $X,Y,\ldots$
to denote smooth vector fields on a smooth manifold $M$, which commute, $[X,Y]=0$.
\end{conv}

Using the vector fields defined in Convention~\ref{Conv-1},
the Levi-Civita connection $\nabla^g$ corresponding to the symmetric non-degenerate (0,2)-tensor $g$ reduces to the following:
\begin{equation}\label{lcg}
g(\nabla^g_XY,Z)=\frac12\big[Xg(Y,Z)+Yg(X,Z)-Zg(Y,X)\big].
\end{equation}

\subsection{Linear Connections on Generalized Riemannian Manifolds}

We consider linear connections $\nabla$ on a smooth manifold $M$
with a torsion (1,2)-tensor
\[
 T(X,Y)=\nabla_XY-\nabla_YX-[X,Y].
\]
We denote the torsion (0,3)-tensor with respect to $g$ by the same letter,
\begin{equation*}
 T(X,Y,Z):=g(T(X,Y),Z).
\end{equation*}

A linear connection
on a generalized Riemannian manifold $(M,G)$ is completely determined by the torsion tensor and the covariant derivative $\nabla g$ of the symmetric part  $g$ of $G$, see \cite{IZ1}.

\begin{dfn}[see \cite{rz-2025}]\rm
A linear connection $\nabla$ on
a generalized Riemannian manifold $(M, G = g + F)$
is said to have {\it A-torsion condition}, where $A$ is given by \eqref{m1}, if its torsion tensor $T$ satisfies
\begin{equation}\label{Eq-A-T}
T(AX, Y) = T(X,AY)  \quad \text{for all } X, Y \in \mathfrak{X}_M .
\end{equation}
A linear connection \( \nabla \) is said to have {\it $Q$-torsion condition} if its torsion tensor \( T \) satisfies
\begin{equation}\label{Eq-Q-T}
T(QX, Y) = T(X, QY)  \quad \text{for all } X, Y \in \mathfrak{X}_M,
\end{equation}
where \( Q: TM \to TM \) is an endomorphism that is self-adjoint with respect to the symmetric part $g$ of metric $G$, i.e. $g(QX, Y) = g(X, QY)$ for all  $X, Y \in \mathfrak{X}_M$.
\end{dfn}

Note that the $Q$-torsion condition is trivial when
$Q={\rm Id}$ (or, $Q$ is conformal: $Q=\lambda\,{\rm Id}$).


\smallskip

The Nijenhuis tensor $N_P$ of a (1,1)-tensor $P$ on a smooth manifold $M$ is defined by (e.g. \cite{KN}),
\begin{equation}\label{nuj}
 N_P(X,Y)=[PX,PY]+P^2[X,Y]-P[PX,Y]-P[X,PY].
\end{equation}
The Nijenhuis tensor is skew-symmetric by definition.
We denote the Nijenhuis (0,3)-tensor with respect to a Riemannian metric $g$ with the same letter,
\[
 N_P(X,Y,Z):=g(N_P(X,Y),Z).
\]

The Nijenhuis tensor $N_A$ plays a fundamental role in almost complex (resp. almost para-complex) geometry.
If~$A^2=-{\rm Id}$ (resp. $A^2={\rm Id}$) then the celebrated Nulander-Nirenberg theorem (see, e.g. \cite{KN}) shows that an almost complex structure is integrable if and only if $N_A$ vanishes.

Using the definition of the torsion tensor $T$ of a linear connection $\nabla$ and the covariant derivative $\nabla A$, we can express the Nijenhuis tensor $N_A$ in terms of $T$ and $\nabla A$ as follows:
\begin{align}\label{nuj1}
\notag
N_A(X,Y)&=(\nabla_{AX}A)Y-(\nabla_{AY}A)X-A(\nabla_{X}A)Y+A(\nabla_{Y}A)X\\
&-T(AX,AY)-A^2T(X,Y)+AT(AX,Y)+AT(X,AY).
\end{align}

For a self-adjoint endomorphism \( Q: TM \to TM \)
we have
\begin{align}\label{nujQ}
\notag
N_Q(X,Y,Z)&=g((\nabla_{QX}Q)Y,Z)-g((\nabla_{QY}Q)X,Z)-g((\nabla_{X}Q)Y,QZ)+g((\nabla_{Y}Q)X,QZ)\\
&-T(QX,QY,Z)-T(X,Y,Q^2Z)+T(QX,Y,QZ)+T(X,QY,QZ).
\end{align}






\subsection{Einstein Metricity Condition}\label{ngt}
In his attempt to construct an unified field theory,
briefly NGT, A.~Einstein \cite{Ein}
considered a generalized Riemannian manifold
$(M,G=g+F)$ with a linear connection $\nabla$ satisfying  the
EMC \eqref{metein-0}, which has the following
coordinate-free form, see \cite{IZ1}:
\begin{equation}\label{metein}
 XG(Y,Z)-G(\nabla_YX,Z)-G(Y,\nabla_XZ)=0
 \quad \text{for all } X, Y, Z \in \mathfrak{X}_M .
\end{equation}
In the paper we will call such linear connection \textit{Einstein's connections}. 
Using the definition of the torsion (0,3)-tensor,
\eqref{metric} and \eqref{m1}, the EMC \eqref{metein} can be presented in the following form, see \cite{IZ1}:
\begin{equation}\label{metein1}
(\nabla_XG)(Y,Z)=-G(T(X,Y),Z) \ \  \Leftrightarrow \ \
(\nabla_X(g+F))(Y,Z)=-T(X,Y,Z)+T(X,Y,AZ).
\end{equation}
Separating symmetric and skew-symmetric parts of \eqref{metein1} (w.r.t. $Y$ and $Z$), we express the covariant derivatives $\nabla g$ and $\nabla F$ in terms of the exterior derivative $dF$ and torsion:
\begin{align}\label{ein6}
(\nabla_X\,g)(Y,Z)
&=-\frac12\big[T(X,Y,Z)+T(X,Z,Y)-T(X,Y,AZ)-T(X,Z,AY)\big],\\
\label{ein5}
\notag
(\nabla_Z\,F)(X,Y)
&= \frac12\big[ T(X,Z,Y)-T(X,Y,Z)+T(X,Y,AZ)-T(X,Z,AY)\big] \\
&=\frac12\big[dF(X,Y,Z)+T(X,Y,Z)-T(Z,Y,AX)+T(Z,X,AY)\big].
\end{align}
Using the vector fields defined in Convention~\ref{Conv-1}, the co-boundary formula for exterior derivative of a 2-form $F$ reduces to the following
formula (without the coefficient 3, unlike \cite{Blair}):
\begin{align}\label{E-3.3}
 d F(X,Y,Z) = X(F(Y,Z)) + Y(F(Z,X)) + Z(F(X,Y))
.
\end{align}
The connection $\nabla$ of \eqref{metein} is represented in \cite{IZ1} as
\begin{multline}\label{genconein}
g(\nabla_XY,Z)=g(\nabla^g_XY,Z)
+\frac12\big[T(X,Y,Z)-T(X,Z,AY)-T(Y,Z,AX)\big]\\
=g(\nabla^g_XY,Z)-\frac12\big[dF(X,Y,Z)+T(Z,X,Y)+T(Y,Z,X)\big]
+\frac12\big[T(Z,X,AY)+T(Z,Y,AX)\big].
\end{multline}



In local coordinates,
the equations (\ref{ein6})-(\ref{genconein}) have the following form, see \cite{IZ1}:
\begin{equation}\label{local}
\begin{aligned}
\nabla_k F_{ij}&=\dfrac 12\big[dF_{ijk}+T_{ijk}-T_{kjs}A^s_i-T_{kis}A^s_j\big],\\
\nabla_i g_{jk}&=-\dfrac 12\big[T_{ijk}-T_{ijs}A^s_k+T_{ikj}-T_{iks}A^s_j\big],\\
\Gamma_{ijk} &=
\Gamma^g_{ijk} + \frac{1}{2} \big[ T_{ijk} - T_{ikp} A^p_j - T_{jkp} A^p_i \big].
\end{aligned}
\end{equation}

The \emph{contorsion (or, difference)} (1,2)-tensor \( K \)
of a linear connection $\nabla$ is defined~by
\begin{equation*}
 K(X, Y) = \nabla_X Y - \nabla^g_X Y.
\end{equation*}
The contorsion $(0,3)$-tensor $K(X,Y,Z)$ is defined by $K(X,Y,Z):=g(K(X,Y),Z)$.

\begin{lemma}
Let 
\( \nabla \) 
be an Einstein's connection 
on a generalized Riemannian manifold $(M,G=g+F)$ satisfy EMC \eqref{metein}. Then the contorsion and torsion $(0,3)$-tensors of $\nabla$ are related as
\[
2\,K(X,Y,Z) =  T(X,Y,Z)-T(X,Z,AY)-T(Y,Z,AX).
\]
\end{lemma}

\begin{proof}
 It follows directly from \eqref{genconein}.
\end{proof}

\subsection{
NGT with Totally Skew-Symmetric Torsion}

Here, we consider a linear connection $\nabla$ with totally skew-symmetric torsion (0,3)-tensor, $T(X,Y,Z)=-T(X,Z,Y)$.
In this case,
\begin{itemize}
 \item[$i)$] the $A$-torsion condition \eqref{Eq-A-T} implies \begin{equation}\label{QAQ}
 T(AX, Y,Z) = T(X, AY,Z) = T(X, Y,AZ).
 \end{equation}
 \item[$ii)$] the $Q$-torsion condition \eqref{Eq-Q-T} implies
\begin{equation}\label{QAQ2}
 T(QX, Y,Z) = T(X, QY,Z) = T(X, Y,QZ).
\end{equation}
\end{itemize}


The following result, see \cite[Theorem~3.1]{IZ1}, presents conditions for the existence and uniqueness of the Einstein's connection on a generalized Riemannian manifold and gives its explicit expression.

\begin{thrm}
\label{pp1}
A generalized Riemannian manifold $(M,G=g+F)$ admits an
Einstein's connection
with totally skew-symmetric torsion $T$ 
if and only if the Nijenhuis tensor $N_A$, the tensor $A$ and the exterior derivative of $F$ satisfy the following relation:
\begin{align}\label{skew1}
\notag
N_A(X,Y,Z)&=\frac23dF(X,Y,AZ)+\frac13dF(AX,Y,Z)+\frac13dF(X,AY,Z)+\frac13dF(AX,AY,AZ)\\
\notag
&-\frac16\big[dF(A^2X,Y,AZ)+dF(A^2X,AY,Z)+dF(X,A^2Y,AZ)
-dF(X,AY,A^2Z)\big]\\
&-\frac16\big[dF(AX,A^2Y,Z)-dF(AX,Y,A^2Z)\big];
\end{align}
moreover, the covariant derivatives of the tensors $F$ and $A$ with respect to
$\nabla^g$ are related by
 \begin{multline}\label{ff2}
(\nabla^g_XF)(Y,Z)=g((\nabla^g_XA)Y,Z)\\
=\frac{1}{3} dF(X,Y,Z)+\frac{1}{3}dF(X,AY,AZ)
-\frac16dF(AX,Y,AZ) -\frac16dF(AX,AY,Z).
\end{multline}

In this case, the totally skew-symmetric torsion $(0,3)$-tensor is completely determined by $dF$:
\begin{equation}\label{tordfnew}
T(X,Y,Z)=-\frac13dF(X,Y,Z),
\end{equation}
the EMC \eqref{metein}
is equivalent to the following two conditions,
see \eqref{ein6}-\eqref{ein5}:
\begin{equation}\label{skew0}
\begin{split}
& (\nabla_Xg)(Y,Z)=-\frac16\big[dF(X,Y,AZ)-dF(X,AY,Z)\big], \\
& (\nabla_XF)(Y,Z)=\frac16\big[2dF(X,Y,Z)-dF(X,Y,AZ)-dF(X,AY,Z)\big],
\end{split}
\end{equation}
and the linear connection $\nabla$ is uniquely determined by the following formula, see \eqref{genconein}:
\begin{equation}\label{newnbl}
g(\nabla_XY,Z)=g(\nabla^g_XY,Z)
  +\frac16\big[ dF(AX,Y,Z) - dF(X,Y,Z)- dF(X,AY,Z) \big].
\end{equation}
\end{thrm}

\begin{rmrk}\rm
By \eqref{newnbl}, the contorsion (0,3)-tensor $K(X,Y,Z)$
of a connection $\nabla$ with a totally skew-symmetric torsion $(0,3)$-tensor is given by the following formula:
\begin{equation}\label{newnbl-K}
 K(X,Y,Z)= \frac16\big[ dF(AX,Y,Z) - dF(X,AY,Z) - dF(X,Y,Z) \big].
\end{equation}
\end{rmrk}

The formulas in Theorem~\ref{pp1} are especially meaningful under the assumption that the torsion
tensor satisfies the $A$-torsion condition (\ref{Eq-A-T}) with totally skew-symmetric property.

\begin{prop}\label{Pr-2.7}
Let $(M, G = g + F)$ be a generalized Riemannian manifold with a
fundamental 2-form $F$. Then an Einstein's connection
\( \nabla \), 
having a totally skew-symmetric torsion \( (0,3) \)-tensor \( T \), satisfies the \( A \)-torsion condition \eqref{Eq-A-T} if and only if it preserves the symmetric part of the metric
\( G \), i.e., $\nabla g=0$.
In this case,
\begin{align}\label{Eq-2.27}
\notag
 & (i)~(\nabla^g_X A)Y = -T(X, Y) \ \Leftrightarrow \
 \nabla^g A = -T,\\
 & (ii)~N_A(X,Y,Z)=\dfrac43\,dF(X,Y,AZ).
\end{align}
\end{prop}

\begin{proof}
Using the first equation in (\ref{skew0}), the equivalence (with $\nabla g=0$) follows.

$(i)$ Using (\ref{QAQ}), from \eqref{ff2} and  \eqref{tordfnew} we obtain
\begin{equation}\label{ISK}
g\big((\nabla^g_X A)Y, Z\big) = \frac{1}{3} dF(X, Y, Z)=-T(X,Y,Z),
\end{equation}
which completes the proof of $(i)$.

$(ii)$ Using (\ref{QAQ}) in \eqref{skew1}, we get the required equation for the Nijenhuis (0,3)-tensor \( N_A \).
\end{proof}

\begin{rmrk}\rm
By \eqref{Eq-2.27}\,(i), we have $(\nabla^g_X A)X = 0$, which  corresponds to weak nearly (para) K\"ahler or weak nearly (para) cosymplectic structures considered in Section~\ref{sec:03}.
\end{rmrk}








\section{Applications of Weak Metric Structures}
\label{sec:03}

In this section, we supply a number of examples (given below using weak metric structures, see \cite{rov-137,rov-108}) with a (1,1)-tensor $A$ of constant rank.
First, we prove the following.

\begin{lemma}\label{L-2.4}
Let a generalized Riemannian manifold $(M, G = g + F)$ be
equipped with
a self-adjoint (with respect to $g$) endomorphism $Q>0$ such that $[Q,A]=0$.
Then at each point $x\in M$ there is a basis
$\{e_1, Ae_1,\ldots, e_m, Ae_m, \xi_1,\ldots,\xi_s\}$
(called an $A$-$Q$-{basis}),
consisting of mutually orthogonal nonzero vectors of $T_xM$
such that $A$ and $Q$ have block-diagonal structures:
$Q = [\lambda_1{\rm Id}_{\,n_1},\ldots,\lambda_k{\rm Id}_{\,n_k}, \nu_1,\ldots,\nu_s]$ and
$A=[\sqrt{\lambda_1}J_{\,n_1},\ldots,\sqrt{\lambda_k}J_{\,n_k},
{\rm 0}_{\,s}]$, where $\lambda_i>0$, $\nu_i\ne0$ and
$J_{\,n_i}$ is a complex structure $(J^2_{\,n_i}=-{\rm Id}_{\,n_i})$ on a $n_i$-dimensional subspace of $T_xM$.
\end{lemma}

\begin{proof}
Suppose that $A$ is non-degenerate at $x\in M$.
Let $e_1\in T_xM$ be a unit eigenvector of the self-adjoint operator $Q>0$ with the minimal eigenvalue $\lambda_1\ne0$. Then, $Ae_1\in T_xM$ is orthogonal to $e_1$ and $Q(Ae_1) = A(Qe_1) = \lambda_1 Ae_1$.
Thus, the subspace of $T_xM$ orthogonal to the plane Span$\{e_1,Ae_1\}$ is $Q$-invariant (and $A$-invariant).
Continuing in the same manner, we find a basis $\{e_1, Ae_1,\ldots, e_m, Ae_m\}$ of $T_x M$
consisting of mutually orthogonal vectors.
Hence, $Q$ has $k$ different nonzero eigenvalues $\lambda_1<\ldots<\lambda_k$ of even multiplicities $n_1,\ldots,n_k$,
and $\sum\nolimits_{\,i=1}^k n_i=2m=\dim M$.
In this basis, $A$ and $Q$ have the required block-diagonal structures.
If $A$ is degenerate at $x\in M$ and $(\ker A)_x$ is
$s$-dimensional, then the proof is similar.
\end{proof}

\begin{rmrk}\rm
The condition \( [A, Q] = 0 \), see Lemma~\ref{L-2.4},
is satisfied by structural tensors of all weak metric structures considered in Section~\ref{sec:03}.
\end{rmrk}

\subsection{Weak Almost Hermitian Structure} \label{sec:2.5.2}

Let us consider a weak almost Hermitian manifold
$M(A, Q, g)$, i.e. a Riemannian manifold $(M, g)$ of dimension $n\,(= 2m \geq 4)$ endowed with
non-singular endomorphisms: $A$ (skew-symmetric) and $Q$ (self-adjoint),
and the fundamental 2-form \(F\)
such that the following conditions are valid, see \cite{rov-137}:
\begin{equation}\label{WAH}
A^2 = -Q, \quad g(AX, AY) = g(QX, Y),
\quad F(X,Y)=g(AX,Y).
\end{equation}
From $A^2=-Q$ we conclude that $A$ commutes with $Q$:
 $[A,Q]=0$;
hence
 $F(X, QY) = F(QX, Y)$.

A.~Gray defined in \cite{G-70} a nearly K\"{a}hler structure $(J,g)$, where $J$ is an almost complex structure, using condition that the symmetric part of $\nabla^g J$ vanishes.

\begin{dfn}\rm
A weak almost Hermitian manifold is said to be \textit{weak nearly K\"ahler} if the
covariant derivative of $A$ (or $F$) with respect to the Levi-Civita connection $\nabla^g$
is skew-symmetric:
\[
(\nabla^g_X A)X=0\quad \Longleftrightarrow \quad (\nabla^g_X F)(X,\,\cdot\,)=0.
\]
If $\nabla^g A=0$, then such $M(A, Q, g)$ is called a \textit{weak K\"ahler manifold}; in this case, $\nabla^g Q=0$.
\end{dfn}

At the same time, $M(A, Q, g)$ admits a generalized Riemannian structure $G=g+F$.

\begin{exm}\label{Ex-3.9a}\rm
Let $M(A, Q, g)$ be a  weak nearly K\"{a}hler manifold
with a fundamental 2-form $F$, conside\-red as a generalized Riemannian manifold \((M, G = g + F)\).
Suppose that $\nabla$ is an Einstein's connection
on $M(A, Q, g)$ with totally skew-symmetric torsion.
Let us show that $\nabla$ satisfies the $A$-torsion condition \eqref{QAQ}.
Since the manifold is
weak nearly K\" ahler, we have
\[
g\big((\nabla^g_X A)Y, Z\big) + g\big((\nabla^g_Y A)X, Z\big) = 0.
\]
On the other hand, using equation~(\ref{ff2}), we obtain the following:
\[
g\big((\nabla^g_X A)Y, Z\big) + g\big((\nabla^g_Y A)X, Z\big) = \frac{1}{2}dF(X, AY, AZ) - \frac{1}{2}dF(AX, Y, AZ),
\]
which leads to the equality
$dF(AX, Y, AZ) = dF(X, AY, AZ)$.
This equation, using (\ref{tordfnew}) and the non-degeneracy of $A$, implies \eqref{Eq-A-T}:
 $T(AX, Y, Z) = T(X, AY, Z)$,
and by the totally skew-symmetry of torsion, we conclude the $A$-torsion condition \eqref{QAQ} is true.
\end{exm}

The following result complements Theorem 3.3 of \cite{IZ1}.

\begin{thrm}\label{mainweak}
Let $M(A, Q, g)$ be a weak almost Hermitian manifold with a fundamental 2-form~\(F\), considered as a generalized Riemannian manifold \((M, G = g + F)\).
Suppose that $\nabla$ is an Einstein's connection
with totally skew-symmetric torsion $(0,3)$-tensor $T$. 
If the $A$-torsion condition \eqref{Eq-A-T} is true,
then we get the following:
\begin{equation}\label{Eq-3.2}
 T(AX,Y,Z) =-\frac{1}{3}\,dF(AX, Y, Z)=-\frac{1}{4}\,N_A(X,Y,Z),
\end{equation}
\begin{align}\label{nabla_Q0}
\nabla Q = \nabla^g Q = 0,
\end{align}
and \(M(A, Q, g)\) is a weak nearly K\"ahler manifold.
\end{thrm}

\begin{proof}

Since the $A$-torsion condition \eqref{Eq-A-T} is true,
the first equality of \eqref{Eq-3.2} follows from \eqref{tordfnew} and the second equality of \eqref{Eq-3.2} follows from Proposition~\ref{Pr-2.7}\,(ii).
Using \eqref{WAH}
and (\ref{ff2}), we get
\begin{align}\label{nabla_g_Q}
\notag
g\big((\nabla^g_X Q)Y, Z\big) &= -g\big((\nabla^g_X A)AY, Z\big) + g\big((\nabla^g_X A)Y, AZ\big) \\
\notag
&= -\frac{1}{3} dF(X, AY, Z) + \frac{1}{3} dF(X, QY, AZ) + \frac{1}{3} dF(X, Y, AZ) \\
&\quad - \frac{1}{3} dF(X, AY, QZ)
-\frac{1}{6} dF(AX, QY, Z)  +\frac{1}{6} dF(AX, Y, QZ).
\end{align}
Using
\eqref{tordfnew} and the $A$-torsion condition
(which implies the $Q$-torsion condition) we obtain:
\begin{align*}
g\big((\nabla^g_X Q)Y, Z\big) &= T(X, AY, Z) -T(X, QY, AZ) -T(X, Y, AZ) \\
&\quad +T(X, AY, QZ)
+\frac{1}{2}T(AX, QY, Z)  -\frac{1}{2} T(AX, Y, QZ)=0.
\end{align*}
Further, using (\ref{newnbl}) and $\nabla^g Q=0$, we get 
\begin{align}\label{nabla_Q}
\notag
g\big((\nabla_X Q)Y, Z\big) &=g\big(\nabla_X QY, Z\big)-g\big(\nabla_X Y, QZ\big)\\
\notag &= g\big((\nabla^g_X Q)Y, Z\big)
+ \frac{1}{6} dF(AX, QY, Z)- \frac{1}{6} dF(X, QY, Z) - \frac{1}{6} dF(X, AQY, Z)  \\
\notag
&\quad - \frac{1}{6} dF(AX, Y, QZ)+ \frac{1}{6} dF(X, Y, QZ) + \frac{1}{6} dF(X, AY, QZ)\\&= \frac{1}{6} dF(AX, QY, Z)- \frac{1}{6} dF(X, Y, QZ) - \frac{1}{6} dF(AX, QY, Z)  \\
\notag
&\quad - \frac{1}{6} dF(AX, Y, QZ)+ \frac{1}{6} dF(X, Y, QZ) + \frac{1}{6} dF(AX, Y, QZ)=0\end{align}
Similarly to the case of $\nabla^g Q=0$, the $A$-torsion condition yields
$\nabla Q=0$.
By Proposition~\ref{Pr-2.7}\,(i), $(\nabla^g_X A)X=-T(X,X)=0$ holds; hence \(M(A, Q, g)\) is a weak nearly K\"ahler manifold.
\end{proof}

\begin{rmrk}\rm
By \cite[Equation (3.7)]{rz-2025} we have
$dF(AX,Y,Z)=3\,T(QX,Y,Z)$.
By \eqref{Eq-3.2}, we have $dF(AX,Y,Z)= -3\,T(AX,Y,Z)$.
The different expressions for $dF(AX,Y,Z)$ arise since we use different connections: we use a metric connection preserving $G$ in \cite{rz-2025}, whereas in this paper we work with Einstein's connections.
In both cases, we assume that the torsion is totally skew-symmetric.
\end{rmrk}


\begin{dfn}\rm
Take two (or more) almost Hermitian manifolds $M_j(A_j, g_j)$, thus $A_j^2=-{\rm Id}_{\,j}$. The~product $\prod_{\,j=1}^k M_j(\sqrt{\lambda_j}\,A_j, g_j)$ of $k$ weak almost Hermitian manifolds, where $\lambda_j>0$ are different constants, is a weak almost Hermitian manifold with $Q=\bigoplus_{\,j}\lambda_j\,{\rm Id}_{\,j}$. We~call $\prod_{\,j} M_j(\sqrt{\lambda_j}\,A_j, g_j)$ a $(\lambda_1,\ldots,\lambda_k)$-\textit{weighed product of almost Hermitian manifolds} $M_j(A_j, g_j)$.
\end{dfn}

The following example represents a set of 
strictly weak nearly K\"{a}hler manifolds.

\begin{exm}
\rm
Note that the $(\lambda_1,\ldots,\lambda_k)$-weighed product of nearly K\"{a}hler manifolds is a weak nearly K\"{a}hler manifold.
A nearly K\"{a}hler manifold of dimension $\le 4$ is a K\"{a}hler manifold, see \cite{G-70}. The~6-dimensional unit sphere in the set of purely imaginary Cayley numbers is an example of a strictly nearly K\"{a}hler manifold.
The classification of weak nearly K\"{a}hler manifolds in dimensions $\ge 4$ is an open problem.
Some 4-dimensional weak nearly K\"{a}hler ma\-nifolds appear as $(\lambda_1,\lambda_2)$-weighed pro\-ducts of 2-dimensional K\"{a}hler manifolds.
Some 6-dimensional weak nearly K\"{a}hler manifolds are $(\lambda_1,\lambda_2,\lambda_3)$-weighed products of 2-dimensional K\"{a}hler manifolds or $(\lambda_1,\lambda_2)$-weighed products of 2- and 4-dimensional K\"{a}hler manifolds. 
Some 8-dimensional weak nearly K\"{a}hler manifolds are 
$(\lambda_1,\lambda_2,\lambda_3,\lambda_4)$-weighed products of 2-dimensional K\"{a}hler manifolds or $(\lambda_1,\lambda_2)$-weighed products of 2-dimensional K\"{a}hler manifolds and 6-dimensional nearly K\"{a}hler manifolds, or $(\lambda_1,\lambda_2)$-weighed products of 4-dimensional nearly K\"{a}hler manifolds, and simi\-larly for even dimensions~$>8$.
The~$(\lambda_1,\ldots,\lambda_k)$-weighed products of
nearly K\"{a}hler manifolds serve as new models for~NGT.
\end{exm}

\begin{thrm}\label{Th-001}
Let $M(A, Q, g)$ be a weak almost Hermitian manifold with a fundamental 2-form \(F\), considered as a generalized Riemannian manifold \((M, G = g + F)\).
Suppose that an Einstein's connection $\nabla$ on $M$
with totally skew-symmetric torsion
satisfies 
the $A$-torsion condition \eqref{Eq-A-T}.

(i)~If $Q = \lambda\,{\rm Id}$ for $\lambda\in C^\infty(M)$, then $\lambda=const>0$ and
$(\lambda^{-1/2}A,g)$ is a nearly K\"{a}hler structure.

(ii)~If $Q\ne\lambda\,{\rm Id}$ for $\lambda\in C^\infty(M)$, then there exist $k>1$ mutually orthogonal even-dimensional distributions
${\mathcal D}_i\subset TM\ (1\le i\le k)$ such that $\bigoplus_{\,i}{\mathcal D}_i=TM$ and ${\mathcal D}_i$ are the eigen-distributions of $Q$
with
constant eigenvalues $\lambda_i:$  $0<\lambda_1<\ldots<\lambda_k$; moreover, each ${\mathcal D}_i$
defines a $\nabla^g$-totally geodesic foliation
and $M(A,Q,g)$ is locally the $(\lambda_1,\ldots,\lambda_k)$-weighed product of nearly K\"{a}hler manifolds.
\end{thrm}

\begin{proof}

(i) By \eqref{nabla_Q0}, we get $\lambda=const>0$, hence
$(\lambda^{-1/2}A,g)$ is an almost Hermitian structure.
Since
the $A$-torsion condition
\eqref{Eq-A-T} is true, by Theorem~\ref{mainweak}, $(\lambda^{-1/2}A,g)$ is a nearly K\"{a}hler structure.

(ii) Since $Q=-A^2$ is not conformal, it has $k>1$ different eigenvalues $0<\lambda_1<\ldots<\lambda_k$ of even multiplicities $n_1,\ldots,n_k$. By Lemma~\ref{L-2.4}, there exists an $A$-$Q$-{basis} at a point $x\in M$,
%
in which $A$ and $Q$ have block-diagonal structures:
$Q = [\lambda_1{\rm Id}_{\,n_1},\ldots,\lambda_k{\rm Id}_{\,n_k}]$ and $A=[\sqrt{\lambda_1}J_{\,n_1},\ldots,\sqrt{\lambda_k}J_{\,n_k}]$, where $J_{\,n_i}$ is a complex structure on a $n_i$-dimensional subspace of~$T_xM$.

Since $Q$ is $\nabla^g$-parallel, see \eqref{nabla_Q0}, we get the same structure at every point of $M$, that is, $k$ and all $\lambda_i$ are constant on $M$, and
there exist mutually orthogonal $\nabla^g$-parallel (and $A$-invariant) eigen-distributions ${\mathcal D}_i$ of $Q$ with constant different eigenvalues $\lambda_i$.
Since the $Q$-torsion condition is true,
using \eqref{Eq-Q-T} for any vector fields $X,Y\in{\mathcal D}_i$, i.e., $QX=\lambda_iX$ and $QY=\lambda_iY$, we have
\begin{align*}
 Q[X,Y] & = Q\{\nabla^g_XY-\nabla^g_YX\}
 = \nabla^g_X(QY)-\nabla^g_Y(QX) \\
 & =\lambda_i\{\nabla^g_XY-\nabla^g_YX\} =\lambda_i[X,Y].
\end{align*}
Hence each ${\mathcal D}_i$ is involutive
and defines a foliation ${\mathcal F}_i$.
Similarly we can show that $Q(\nabla^g_XY)=\nabla^g_X(QY)=\lambda_i\nabla^g_XY$,
hence ${\mathcal F}_i$ is a $\nabla^g$-totally geodesic foliation,
and by de Rham Decomposition Theorem (see \cite{KN}), 
our manifold splits and is the
$(\lambda_1,\ldots,\lambda_k)$-weighed product of
almost Hermitian manifolds.
By Theorem~\ref{mainweak}, the factors are
nearly K\"{a}hler~manifolds.
\end{proof}

\begin{exm}\rm
Let a generalized Riemannian manifold \((M, G = g + F)\)
be represented as the $(\lambda_1,\ldots,\lambda_k)$-weighed product of nearly K\"{a}hler manifolds
$M_j(A, g_j)\ (1\le j\le k)$.
We get a weak nearly K\"{a}hler structure on $M$ with $Q=\bigoplus_{\,j}\lambda_j\,{\rm Id}_{\,j}$ for some constants $\lambda_j>0$, hence $\nabla^g Q=0$.

By \cite[Theorem~3.3]{IZ1}, for any $j$ there exists a unique Einstein's connection $\nabla^{(j)}$ on a nearly K\"{a}hler manifold $M_j(A_j, g_j)$. Its torsion is determined by
\eqref{Eq-3.2}.
%
It was shown in Example~\ref{Ex-3.9a} that
%
this $\nabla^{(j)}$ is an Einstein's connection on the weak nearly K\"{a}hler manifold
$M_j(\sqrt{\lambda_j}A_j,\,\lambda_j\,{\rm Id}_{\,j},\,g_j)$ satisfying the $A$-torsion condition \eqref{Eq-A-T}.
A unique linear connection \(\nabla\) on \((M, G = g + F)\) with a totally skew-symmetric torsion satisfying EMC \eqref{metein} and the $A$-torsion condition \eqref{Eq-A-T}
is the metric connection, i,e. $\nabla g=0$, its torsion is $T=-\nabla^g A$, see Proposition~\ref{Pr-2.7}\,(i),
and $\nabla Q=0$, see \eqref{nabla_Q0}.
By the above, this Einstein metric connection has the following~form:
$\nabla=\bigoplus_{\,j}\nabla^{(j)}$.
\end{exm}

\subsection{Weak Almost Contact Metric Structure}

Contact Riemannian geomet\-ry is of growing interest due to its important role in both theoretical physics and pure mathematics. Weak a.c.m. structures, i.e., the complex structure on the contact distribution is approximated by a non-singular skew-symmetric tensor, allowed us to take a new look at the theory of contact manifolds and find new applications.

\begin{dfn}\rm
A \textit{weak a.c.m. manifold} $M(A, Q, \xi, \eta, g)$ is a $(2m+1)$-dimensional Riemannian manifold equipped with a skew-symmetric (1,1)-tensor $A$ of rank $2m$,
a unit vector field $\xi$, a 1-form $\eta$ dual to $\xi$ with respect to the metric $g$,\ $\eta(\xi)=1,\ \eta(X)=g(X,\xi)$,
and a self-adjoint (1,1)-tensor $Q>0$, satisfying the following compatibility conditions:
\begin{align}\label{acon}
A^2=-Q+\eta\otimes\xi, \quad
g(AX,AY)=g(QX,Y)-\eta(X)\eta(Y), \quad
A\xi=0, \quad
Q\xi=\xi.
\end{align}
A weak a.c.m. manifold $M(A,Q,\xi,\eta,g)$ is said to be \textit{weak almost-nearly cosymplectic}
if it satisfies the following condition:
\begin{align}\label{a-nearly-cos}
g((\nabla^g_X A)Y, Z) = -\frac{1}{3} dF(AX, AY, Z) + \frac{1}{6} \eta(Z) d\eta(Y, AX) - \frac{1}{2} \eta(Y) d\eta(AZ, X).
\end{align}
\end{dfn}


From (\ref{acon}) we conclude that $A$ commutes with $Q$:
 $[A,Q]=0$; hence $F(X, QY) = F(QX, Y)$.

If we assume $d\eta=0$, then \eqref{a-nearly-cos} reduces to
\begin{align}\label{near-eta}
 g((\nabla^g_X A)Y, Z) = -\frac{1}{3} dF(AX, AY, Z),
\end{align}
and a weak almost-nearly cosymplectic manifold becomes weak nearly cosymplectic: $(\nabla^g_X A)X=0$.

The following lemma generalizes \cite[Corollary~3.9]{IZ1}.

\begin{lemma}\label{L-Killing}
 The Reeb field $\xi$  of a {weak almost-nearly cosymplectic manifold} $M(A,Q,\xi,\eta,g)$ is
 a geodesic vector field, i.e., $\nabla^g_\xi\,\xi=0$,
 and a Killing vector field, i.e.,
 $g(\nabla^g_X\,\xi,\, Y) + g(\nabla^g_Y\,\xi,\, X)=0$.
\end{lemma}

\begin{proof}
Replacing $Y$ by $\xi$ in \eqref{a-nearly-cos}, we obtain
\begin{align}\label{E-13C}
 g(\nabla^g_X\,\xi, AZ) = \frac{1}{6}\,\eta(Z)\,d\eta(\xi, AX)
 +\frac{1}{2}\,d\eta(X,AZ).
\end{align}
Replacing $X$ by $\xi$ in the above equation and using
$g(\nabla^g_X\,\xi, \xi)=0$, we obtain
\begin{align}\label{E-L3-13}
 g(\nabla^g_\xi\,\xi, Y) = \frac{1}{2}\,d\eta(\xi, Y).
\end{align}
Using the identity
\begin{align}\label{E-deta-nabla}
 d\eta(X,Y)=g(\nabla^g_X\,\xi, Y)-g(\nabla^g_Y\,\xi, X)
\end{align}
with $X=\xi$ (without the coefficient 2, unlike \cite{Blair})
in \eqref{E-L3-13}, we
get $d\eta(\xi, Y)=g(\nabla^g_\xi\,\xi, Y)$.
Comparing with \eqref{E-L3-13}, we conclude that $\xi$ is a geodesic vector field:
\begin{align}\label{E-L3-13e}
 g(\nabla^g_\xi\,\xi, Y) = 0,\quad d\eta(\xi, Y) = 0.
\end{align}
Representing any vector $Y$ as $Y-\eta(Y)\,\xi=AZ$
and using \eqref{E-L3-13} and \eqref{E-L3-13e},
gives
\begin{align}\label{E-L3-13b}
 g(\nabla^g_X\,\xi, Y) = g(\nabla^g_X\,\xi, AZ)
 = \frac{1}{2}\,d\eta(X,AZ) = \frac{1}{2}\,d\eta(X,Y) .
\end{align}
Using \eqref{E-L3-13b}, we have
 $g(\nabla^g_X\,\xi,\, Y) + g(\nabla^g_Y\,\xi,\, X) = 0$.
Therefore, $\xi$
is a Killing vector field.
\end{proof}

\begin{prop}\label{T-2.13}
Let $M(A, Q, \xi, \eta, g)$ be a weak a.c.m. manifold considered as a generalized Riemannian manifold \((M, G = g + F)\), and $\nabla$ an Einstein's connection
with totally skew-symmetric torsion.
Then
$d\eta(X,\xi)=0$ and $\nabla^g_\xi\,\xi=0$ hold,
i.e., the Reeb vector field $\xi$ is a geodesic vector~field.
\end{prop}

\begin{proof}
Using the fact that $g(\xi,\xi)=1$,
$g(\nabla^g_X\,\xi, \xi)=0$, \eqref{acon} and \eqref{ff2}, we have
\begin{align}\label{nwac2}
 (\nabla^g_X\eta)(QZ)=g(\nabla^g_X\,\xi, QZ)
 = -g((\nabla^g_XA)\xi, AZ)
 =\dfrac 13 dF(X,AZ,\xi)+\dfrac 16 dF(AX,QZ,\xi).
\end{align}
Taking $X=\xi$ in \eqref{nwac2}, yields $(\nabla^g_\xi\,\eta)(QZ)=g(\nabla^g_\xi\,\xi, QZ)=0$ for all $Z\in\mathfrak{X}_M$.
By this, since $Q$ is non-degenerate,
$\nabla^g_\xi\,\eta=\nabla^g_\xi\,\xi=0$ and $d\eta(X,\xi)=0$ hold, hence, $\xi$ is a $\nabla^g$-geodesic vector field.
\end{proof}


Let us introduce the (0,3)-tensor, see \cite{rz-2025},
\begin{equation}\label{nwac}
N_A^{\mathrm{wac}} = N_A + d\eta \otimes \eta,
\end{equation}
called the \emph{Nijenhuis tensor in the weak a.c.m. geometry}.

\begin{prop}[see \cite{rov-137}]
For a weak a.c.m. structure $({A},Q,\xi,\eta,g)$, we get
\begin{equation}\label{MAINW}
\begin{aligned}
2\,g\big((\nabla^g_X A)Y, Z\big) &= N^{(5)}(X, Y, Z)
+ dF(X, Y, Z) -dF(X, AY, AZ) + N_A^{\mathrm{wac}}(Y, Z, AX) \\
& + \big[\, d\eta(AY, Z) - d\eta(AZ, Y) \,\big] \eta(X)
- d\eta(X, AY) \eta(Z) + d\eta(X, AZ) \eta(Y)
\end{aligned}
\end{equation}
where the skew-symmet\-ric with respect to $Y$ and $Z$
tensor ${N}^{\,(5)}(X,Y,Z)$ is defined by
\begin{eqnarray}\label{E-N5}
\nonumber
  {N}^{\,(5)}(X,Y,Z)& = ({A} Z)\,(g(X, \widetilde QY)) -({A} Y)\,(g(X, \widetilde QZ)) +g([X, {A} Z], \widetilde QY)\\
 & -\,g([X,{A} Y], \widetilde QZ) +g([Y,{A} Z] -[Z, {A} Y] - {A}[Y,Z],\ \widetilde Q X),
\end{eqnarray}
using the tensor $\widetilde Q=Q-{\rm Id}$.
For particular values of the tensor ${N}^{\,(5)}$ we get
\begin{eqnarray*}
\nonumber
 && {N}^{\,(5)}(X,\xi,Z)
 = g( [\xi, AZ] - A[\xi, Z],\, \widetilde Q X)
 =\frac12\,g((\pounds_\xi A)Z,\,\widetilde Q X),\\
\nonumber
 && {N}^{\,(5)}(\xi,Y,Z) = g([\xi, {A} Z], \widetilde QY)
 -g([\xi,{A} Y], \widetilde QZ) ,\quad
 {N}^{\,(5)}(\xi,\xi,Z)
 =0.
\end{eqnarray*}
\end{prop}

\begin{prop}\label{L-3.16}
Let $M(A, Q, \xi, \eta, g)$ be a {weak almost-nearly cosymplectic manifold} with a fundamental 2-form $F$, conside\-red as a generalized Riemannian manifold \((M, G = g + F)\).
Suppose~that $\nabla$ is an Einstein's connection with totally skew-symmetric torsion. Then the manifold is  weak nearly cosymplectic, and locally is the metric product of a real line and a weak nearly K\"{a}hler~manifold.
\end{prop}

\begin{proof}
From \eqref{E-L3-13b} and \eqref{ff2}, we have
\begin{equation}\label{etaaaa}
\begin{aligned}
 d\eta(X, AZ) &= 2\,g(\nabla^g_X\,\xi, A Z)
 = -2\,g(A\nabla^g_X\,\xi, Z)
 = 2\,g\big((\nabla^g_X A)\xi, Z\big) \\
 & =- \frac{2}{3} dF(X,Z,\xi) + \frac{1}{3} dF(AX, A Z, \xi).
\end{aligned}
\end{equation}
So, for any vector fields $X$ and  $Z$, we have
\begin{equation}\label{deeta32}
 d\eta(X,AZ)=d\eta(AX,Z).
\end{equation}
Comparing \eqref{ff2} and \eqref{a-nearly-cos} with $X$ or $Y$ or $Z$ equal to $\xi$, and using \eqref{E-L3-13e}, we get, respsctively,
\begin{align}\label{E-dF-X}
 dF(\xi, AY, AZ) & = -dF(\xi, Y, Z), \\
\label{E-dF-Y}
 2\,dF(X,\xi, Z) & = dF(AX, \xi, AZ) + 3\,d\eta(X, AZ),\\
\label{E-dF-Z}
 2\,dF(X, Y, \xi) & = - dF(AX, AY, \xi) - d\eta(AX, Y).
\end{align}
From \eqref{E-dF-Y} and \eqref{E-dF-Z} we find
\begin{align*}
 dF(AX, \xi, AY) + 3\,d\eta(X, AY)
 = 2\,dF(X,\xi, Y) =dF(AX, AY, \xi) + d\eta(AX, Y),
\end{align*}
hence, using \eqref{deeta32}, we get
\begin{equation}\label{E-DF-deta}
dF(AX, AY, \xi)=d\eta(X, AY).
\end{equation}
Applying this in \eqref{E-dF-Z} and keeping in mind \eqref{deeta32}, yields
\begin{equation}\label{deta233}
 2\,dF(X, Y, \xi)=-d\eta(X, AY)-d\eta(AX, Y)
 =-2\,d\eta(AX,Y).
\end{equation}
Next, we calculate
\begin{align*}
 d\eta(QY, AZ)= dF(AY,AZ,\xi) = dF(Y,Z,\xi) = -d\eta(AY, Z)
 = -d\eta(Y, AZ),
\end{align*}
hence
\begin{align}\label{detaYQY}
 d\eta(Y+QY, AZ)=0.
\end{align}
Since $Q$ is self-adjoint and positive definite, from \eqref{detaYQY}, using $d\eta(\xi,\,\cdot)=0$ of \eqref{E-L3-13e}, we get $d\eta=0$.
Hence the distribution ${\mathcal D}$ is involutive, i.e., tangent to a codimension-one foliation ${\mathcal F}$.
In view of \eqref{tordfnew}, $T(X, Y, \xi)=0$ and \eqref{near-eta} are true.
Since $\xi$ is a Killing vector field (see Lemma~\ref{L-Killing}), using \eqref{E-deta-nabla} we find $g(\nabla^g_X\,\xi, Y)=0$ for all $X,Y\in\mathfrak{X}_M$,
hence $\xi$ is a $\nabla^g$-parallel vector field: $\nabla^g\,\xi=0$.
Thus, the foliation ${\mathcal F}$ is totally geodesic.
By de Rham Decomposition Theorem, the manifold is weak nearly cosymplectic and locally is the metric product
$\mathbb{R}\times\bar M^{2m}$ of a real line and a weak nearly K\"{a}hler manifold.
\end{proof}

Therefore we generalize \cite[Theorem 3.8]{IZ1} as follows.

\begin{thrm}\label{T-3.8IZ}
Let $M(A, \xi, \eta, g)$ be an almost-nearly cosymplectic manifold with a fundamental 2-form $F$, conside\-red as a generalized Riemannian manifold \((M, G = g + F)\).
Then an Einstein's connection $\nabla$ has
a totally skew-symmetric torsion 
if and only if the manifold is nearly cosymplectic and locally is the metric product of a real line and a nearly K\"{a}hler~manifold. The~torsion is determined by the condition
\[
 T(X,Y,Z) = -\frac13\,dF(X,Y,Z)= -\frac14\,N_A(AX,AY,AZ),
\]
the connection $\nabla$ is uniquely determined by the formula
\[
 g(\nabla_X Y, Z) = g(\nabla^g_X Y, Z) -\frac16\,dF(X,Y,Z),
\]
the covariant derivative of $g$ vanishes: $\nabla g = 0$,
and the covariant derivative of $F$ is
\begin{align*}
 (\nabla_X F)(Y,Z) = \frac13\big\{dF(X,Y,Z)-dF(X,Y,AZ)\big\}.
\end{align*}
\end{thrm}

\begin{proof}
This follows from \cite[Theorem 3.8]{IZ1} and our Proposition~\ref{L-3.16}.
\end{proof}

\begin{exm}\label{Ex-3.9B}\rm
Let $M(A, Q, \xi, \eta, g)$ be a  weak nearly cosymplectic manifold with a fundamental 2-form $F$, considered as a generalized Riemannian manifold \((M, G = g + F)\).
Suppose that $\nabla$ is an Einstein's connection
on $M$ with totally skew-symmetric torsion.
Let us show (similarly, to Example \ref{Ex-3.9a}) that $\nabla$ satisfies the $A$-torsion condition \eqref{QAQ}.

Using equation~(\ref{ff2}), we obtain the following:
\[
0=g\big((\nabla^g_X A)Y, Z\big) + g\big((\nabla^g_Y A)X, Z\big) = \frac{1}{2}dF(X, AY, AZ) - \frac{1}{2}dF(AX, Y, AZ),
\]
which leads to the equality
\[
dF(AX, Y, AZ) = dF(X, AY, AZ).
\]
This equation, using (\ref{tordfnew}), $dF(X, Y, \xi)=0$
(see the proof of Proposition~\ref{L-3.16})
and the non-degeneracy of $A$ on ${\mathcal D}$, implies \eqref{Eq-A-T}:
\[
 T(AX, Y, Z) = T(X, AY, Z).
\]
By the totally skew-symmetry of torsion, we conclude that the $A$-torsion condition \eqref{QAQ} is true.
\end{exm}


\begin{thrm}\label{T-3.18}
Let \( M(A, Q, \xi, \eta, g) \) be a weak a.c.m.
manifold, considered as a generalized Riemannian manifold \( (M, G = g + F) \), and 
$\nabla$ is an Einstein's connection
with totally skew-symmetric torsion satisfying 
the \( A \)-torsion condition \eqref{QAQ}. Then the following properties hold:
\begin{enumerate}
 \item[(i)] The contact distribution ${\mathcal D}=\ker\eta$ is involutive  and
 the Reeb vector field \( \xi \) is parallel with respect to the Levi-Civita connection.
 \item[(ii)] The tensor \( N^{wac}_A\), defined by~\eqref{nwac}, is totally skew-symmetric and $N^{wac}_A(\cdot\,,\cdot\,,\xi) = 0$.
 \item[(iii)] The tensor \( N^{(5)} \), defined by~\eqref{E-N5}, is totally skew-symmetric and is given by
 \begin{align}\label{N51}
 N^{(5)}(X,Y,Z)=-\frac{1}{3} dF(X,Y,Z)-\frac{1}{3} dF(X,AY,AZ).
 \end{align}
 \item[(iv)] The tensor $Q$ satisfies the following equalities:
  $\nabla^g Q =\nabla Q =0$.
\end{enumerate}
\end{thrm}

\begin{proof} From \eqref{nwac2}, assuming the $A$-torsion condition, we obtain
\begin{equation}\label{etaQ}
\begin{aligned}
(\nabla^g_X\eta)(QZ) &= \frac{1}{3} dF(X, AZ, \xi) + \frac{1}{6} dF(AX, QZ, \xi) \\
&= \frac{1}{3} dF(X, Z, A\xi) + \frac{1}{6} dF(X, QZ, A\xi) = 0\quad\text{for all}\;X,Z\in\mathfrak{X}_M.
\end{aligned}
\end{equation}
Restricting to \(Z\in\mathcal D=\ker\eta\), 
we obtain \(A^2Z=-QZ\); hence, setting \(U=AZ+\xi\) yields
\[
AU=A(AZ+\xi)=A^2Z+A\xi=-QZ.
\]
Therefore, \(QZ=-AU\) lies in ${\mathcal D}$. 
By (\ref{etaQ}), $(\nabla^g_X\eta)(QZ)$ vanishes for every \(QZ\) with \(Z\in\mathcal D\), and since \(Q\) is invertible on \(\mathcal D\), it follows that
\[
(\nabla^g_X\eta)(W)=0\qquad\text{for all }W\in\mathcal D.
\]
Finally, \(\eta(\xi)=1\) implies \((\nabla^g_X\eta)(\xi)=0\), so the above equation extends to all \(W\in\mathfrak{X}_M\). Thus,
\begin{equation}\label{nwac3}
g(\nabla^g_X\,\xi,W)=(\nabla^g_X\,\eta)(W)=0\quad\text{for all}\; X,W\in \mathfrak{X}_M,
\end{equation}
hence \(\nabla^g_X\,\xi=0\) for every \(X\), i.e. the Reeb vector field \(\xi\) is parallel with respect to the Levi-Civita connection.
From \eqref{nwac3}, we obtain
\begin{equation}\label{E-deta0}
d\eta(X,Y) = 0 \quad\text{for all}\; X, Y \in \mathfrak{X}_M,
\end{equation}
which implies that the contact distribution ${\mathcal D} = \ker\eta$ is involutive. 
This completes the proof of (i).

Taking into account 
\eqref{E-deta0}, we observe that
\[
 N^{\,\mathrm{wac}}_A(X, Y, Z) = N_A(X, Y, Z).
\]
Substituting the expression~\eqref{ff2} into~\eqref{MAINW}, we then obtain
\begin{align}
N^{wac}_A(Y,Z,AX) ={}& -\dfrac{1}{3} dF(X,Y,Z) + dF(X,AY,AZ) - N^{(5)}(X,Y,Z) \nonumber\\
& +\dfrac 23 dF(X,AY,AZ)- \dfrac{1}{3} dF(AX,Y,AZ) - \dfrac{1}{3} dF(AX,AY,Z).
\end{align}
In view of the $A$-torsion condition, the above equation becomes
\begin{align}\label{NAA}
N^{wac}_A(Y,Z,AX) ={}& -\dfrac{1}{3} dF(X,Y,Z) + dF(X,AY,AZ) - N^{(5)}(X,Y,Z).
\end{align}
By setting $X=\xi$ in the last equation, we obtain
\[
 N^{(5)}(\xi, Y, Z) = -\dfrac13 dF(\xi, Y, Z) =T(\xi,Y,Z).
\]
Applying (\ref{acon}) and (\ref{E-deta0}), we simplify (\ref{skew1}) to get
\begin{align}\label{nuja}
N_A(X,Y,Z) &= \frac{2}{3} dF(X,Y,AZ) + \frac{1}{3} dF(AX,Y,Z) + \frac{1}{3} dF(X,AY,Z)+ \frac{1}{3} dF(AX,AY,AZ)  \notag \\
&\quad + \frac{1}{6} \big[dF(QX,Y,AZ)+dF(QX,AY,Z)+dF(X,QY,AZ)-dF(X,AY,QZ) \big]\notag \\
&\quad+\dfrac 16\big[dF(AX,QY,Z)-dF(AX,Y,QZ)\big].
\end{align}

From equation~(\ref{nuja}), and using the \(A\)-torsion condition together with its consequence
\[
dF(AX, AY, AZ) = -dF(AX, QY, Z) = -dF(AX, Y, QZ)=-dF(QX,AY,Z),
\]
we obtain
\begin{equation}\label{skewacB1}
\begin{aligned}
N^{wac}_A(X,Y,Z) &=\frac{4}{3} dF(AX,Y,Z),
\end{aligned}
\end{equation}
which shows that the Nijenhuis tensor $N^{wac}_A$ is totally skew-symmetric and
$$
 N^{wac}_A(X,Y,\xi)=N^{wac}_A(\xi,X,Y)=0,
$$
that completes the proof of (ii).

Substituting (\ref{skewacB1}) in (\ref{NAA}), we obtain \eqref{N51}, that completes the proof of (iii).

Using the first equation of \eqref{acon},
\( A^2 = -Q+\eta\otimes\xi \), we get
\[
g\big((\nabla^g_X Q)Y, Z\big)
= -g\big((\nabla^g_X A)A Y, Z\big)
+ g\big((\nabla^g_X A) Y, AZ\big)
+ (\nabla^g_X \eta)(Y)\,\eta(Z)
+ \eta(Y)\,g\big(\nabla^g_X\,\xi, Z\big).
\]
From the property (i) ($\nabla^g\xi=\nabla^g \eta=0$) and (\ref{ff2}), we obtain
\begin{align*}
g\big((\nabla^g_X Q)Y, Z\big)
&= -\frac{1}{3} dF(X, AY, Z) + \frac{1}{3} dF(AX, QY, Z)
- \frac{1}{3} \eta(Y) \, dF(X, \xi, AZ) \\
&\quad +\dfrac 16dF(AX,AY,AZ)-\dfrac 16dF(AX,QY,Z)+\dfrac 16\eta(Y)dF(AX,\xi, Z) \\
&\quad +\frac{1}{3} dF(X, Y, AZ) -\dfrac 13dF(X,AY,QZ)+\dfrac 13\eta(Z)dF(X,AY,\xi)\\
&\quad + \frac{1}{6} dF(AX,Y,QZ)-\dfrac 16\eta(Z)dF(AX,Y,\xi)-\dfrac 16dF(AX,AY,AZ) \\
&=-\frac{1}{3} dF(X, AY, Z) +\frac{1}{3} dF(X, Y, AZ) + \frac{1}{6} dF(AX, QY, Z) -\dfrac 13dF(X,AY,QZ)
\\
&\quad + \frac{1}{6} dF(AX,Y,QZ) + \frac{1}{3} \big[\eta(Y) \, dF(X, AZ,\xi) +\eta(Z)dF(X,AY,\xi)\big] \\
&\quad -\dfrac 16\big[\eta(Z)dF(AX,Y,\xi)+\eta(Y)dF(AX, Z,\xi)\big].
\end{align*}
Under the assumption of the $A$-torsion condition \eqref{QAQ}, and after rearranging the terms in the preceding equation, we obtain:
\begin{align}\label{Eq-nabla-g-Q}
\notag
g\big((\nabla^g_X Q)Y, Z\big)
&=-\frac{1}{3} dF(X, AY, Z) +\frac{1}{3} dF(X, AY, Z)
+ \frac{1}{6} dF(AX, QY, Z) -\dfrac 13dF(X,AY,QZ)\\
\notag
&\quad + \frac{1}{6} dF(X,AY,QZ)  + \frac{1}{3} \big[\eta(Y) \, dF(X, Z,A\xi) +\eta(Z)dF(X,Y,A\xi)\big] \\
\notag
&\quad -\dfrac 16\eta(Z)\big[\eta(Z)dF(X,Y,A\xi)+\eta(Y)dF(X, Z,A\xi)\big].\\[1ex]
&=0.
\end{align}
Similarly  we obtain
$\nabla Q=0$,
that completes the proof of (iv).
\end{proof}

The following theorem is similar to Theorem~\ref{Th-001}.

\begin{thrm}\label{Th-3.13}
Let conditions of Theorem~\ref{T-2.13} be satisfied.
Then the following properties are~true.

(i)~If $Q|_{\,\mathcal D}=\lambda\,{\rm Id}_{\mathcal D}$ for some $\lambda\in C^\infty(M)$, then $\lambda=const>0$ and 
$M(\lambda^{-1/2}A,\xi,\eta,g)$ is locally the product of $\mathbb{R}$ and
a nearly K\"{a}hler manifold.

(ii)~If~$Q|_{\,\mathcal D}\ne\lambda\,{\rm Id}_{\mathcal D}$ where $\lambda\in C^\infty(M)$, then there exist $k>1$ mutually orthogonal even-dimensional distributions
${\mathcal D}_i\subset{\mathcal D}$ such that $\bigoplus_{\,i=1}^k{\mathcal D}_i={\mathcal D}$ and ${\mathcal D}_i$ are eigen-distributions of $Q$ with constant eigenvalues 
$0<\lambda_1<\ldots<\lambda_k$;
moreover,
the distributions ${\mathcal D}_i$ are involutive and define $\nabla^g$-totally geodesic foliations
and $M(A,Q,\xi,\eta, g)$ is locally a $(1,\lambda_1,\ldots,\lambda_k)$-weighed product of a real line and $k$ nearly K\"{a}hler manifolds.
\end{thrm}

\begin{proof}
(i) Since $\nabla^g Q=0$, see Theorem~\ref{T-2.13}\,iii), we get $\lambda=const>0$ (and $\lambda\ne1$), hence
$(\lambda^{-1/2}A, \xi, \eta, g)$ is an a.c.m.
structure.
Since
$\nabla$ satisfies the $A$-torsion condition, then using \eqref{Eq-A-T} for any vector fields $X,Y\in{\mathcal D}$, i.e., $QX=\lambda X$ and $QY=\lambda Y$ with $\lambda\ne1$, we have
\begin{align*}
 Q[X,Y] & = Q\{\nabla^g_XY-\nabla^g_YX\}
 = \nabla^g_X(QY)-\nabla^g_Y(QX) \\
 & =\lambda\{\nabla^g_XY-\nabla^g_YX\} =\lambda[X,Y].
\end{align*}
Hence the contact distribution ${\mathcal D}$ is involutive and defines a foliation ${\mathcal F}$. Similarly we show that $Q(\nabla^g_XY)=\nabla^g_X(QY)=\lambda\nabla^g_XY$,
that is, ${\mathcal F}$ is $\nabla^g$-totally geodesic. 
By Theorem~\ref{T-2.13}, $\xi$ is a $\nabla^g$-geodesic vector field. By de Rham Decomposition Theorem (see \cite{KN}),
$M(\lambda^{-1/2}A, \xi, \eta, g)$ splits and is locally the product of $\mathbb{R}$ and an almost Hermitian manifold.
Since $\nabla$ satisfies the $A$-torsion condition, the second factor is a nearly K\"{a}hler~manifold.

(ii) Sinse $Q|_{\,\mathcal D}$ is not conformal,
it has eigenvalues $0<\lambda_1<\ldots<\lambda_k$ of even multiplicities $n_1,\ldots,n_k$, and $\sum\nolimits_{\,i=1}^k n_i= 2m$.
By Lemma~\ref{L-2.4}, there exists an $A$-$Q$-{basis}
of $T_x M$,
in which $A$ and $Q$ restricted on ${\mathcal D}_x$ have block-diagonal forms:
$A|_{\mathcal D}= [\sqrt{\lambda_1}J_{n_1},\ldots,\sqrt{\lambda_k}J_{n_k}]$
and $Q|_{\mathcal D} = [\lambda_1{\rm Id}_{n_1},\ldots,\lambda_k{\rm Id}_{n_k}]$,
where $J_{n_i}$ is a complex structure on a $n_i$-dimensional subspace of~${\mathcal D}_x$.
Since $Q$ is $\nabla^g$-parallel, we get the same structure at each point of $M$, i.e., $k$ and all $\lambda_i$ are constant on $M$, and
there exist mutually orthogonal $\nabla^g$-parallel
(and $A$-invariant)
eigen-distributions ${\mathcal D}_i\subset{\mathcal D}$ of $Q$ with constant different eigenvalues $\lambda_i$.
The rest of the proof is similar to the proof of Theorem~\ref{Th-001}.
\end{proof}

\begin{exm}\rm
Let a generalized Riemannian manifold \((M, G = g + F)\)
be represented as the $(1,\lambda_1,\ldots,\lambda_k)$-weighed product of $\mathbb{R}$ and $k$ nearly K\"{a}hler manifolds
$M_j(A, g_j)\ (1\le j\le k)$.
We~get a weak nearly K\"{a}hler structure on ${\mathcal D}$ with $Q|_{\mathcal D}=\bigoplus_{\,j}\lambda_j\,{\rm Id}_{\,j}$ for some $\lambda_j\in\mathbb{R}_+$, hence $\nabla^g Q=0$.

By \cite[Theorem~3.8]{IZ1}, for any $j$ there exists a unique Einstein's connection $\nabla^{(j)}$ on a nearly K\"{a}hler manifold $M_j(A, g_j)$. Its torsion is determined by the condition \eqref{Eq-3.2}, which is invariant under the change $A \to \lambda A$. This $\nabla^{(j)}$ is also an Einstein's connection on the weak nearly K\"{a}hler manifold
$M_j(\sqrt{\lambda_j}\,A,\,\lambda_j\,{\rm Id}_{\,j},\,g_j)$ satisfying the $A$-torsion condition \eqref{Eq-A-T} with $A=\sqrt{\lambda_j}\,A$.

A unique linear connection \(\nabla\) on \((M, G = g + F)\) with a totally skew-symmetric torsion satisfying EMC \eqref{metein} and the $A$-torsion condition \eqref{Eq-A-T}
is the metric connection, i,e. $\nabla g=0$, its torsion is $T=-\nabla^g A$, see Proposition~\ref{Pr-2.7}\,(i),
and $\nabla Q=0$, see \eqref{nabla_Q0}.
By the above, this Einstein metric connection on $M$ has the following~form:
$\nabla=\bigoplus_{\,j}\nabla^{(j)}$.
\end{exm}

\begin{rmrk}[Weak Para-Hermitian and Weak Almost Para-Contact Structures]\rm
A~\textit{weak almost para-Hermitian manifold}
$M(A, Q, g)$, is a (pseudo-) Riemannian mani\-fold $(M, g)$ of dimension $n\,(= 2m \ge 4)$ endowed with  non-singular endomorphisms: $A$ (skew-symmetric) and $Q>0$ (self-adjoint)
and the fundamental 2-form \(F\), such that the following conditions hold:
\begin{equation}\label{APH}
A^2=Q, \quad g(AX,AY)=-g(QX,Y),\quad F(X,Y)=g(AX,Y).
\end{equation}
In this case, the skew-symmetric part $F$ of $G=g+F$ is non-degenerate and rank$F=2m$.




A \textit{weak almost para-contact metric manifold} 
$M(A, Q, \xi, \eta, g)$ is a $(2m+1)$-dimensional pseudo-Riemannian manifold of signature $(m+1,m)$ equipped with a (1,1)-tensor $A$ of rank $2m$, a vector field $\xi$, a 1-form $\eta$ dual to $\xi$ with respect to the metric $g$,
and a self-adjoint (1,1)-tensor $Q>0$, satisfying the following conditions:
\begin{equation}\label{apcon}
A^2=Q-\eta\otimes\xi, \ \
g(AX,AY)=-g(QX,Y)+\eta(X)\eta(Y), \ \ A\xi=0, \ \ Q\xi=\xi.
\end{equation}
In this case, the skew-symmetric part $F$ of $G=g+F$ is degenerate, $F(\xi,X)=0$, and rank$F=2m$.


By applying analogous technique, one obtains results similar to Theorems~~\ref{mainweak}, \ref{Th-001}, 
\ref{T-3.8IZ}, \ref{T-3.18}, and \ref{Th-3.13}
 in the cases of weak almost para-Hermitian and weak almost para-contact manifolds. Para cases indicating the sign changes type behaviour of the structure endomorphism $Q$. These sign changes  do not significantly affect the essence of the results. For conciseness, their statements and proofs are omitted. 
\end{rmrk}

\section{Conclusion}

{
The paper presents new applications of weak contact metric structures to generalized Riemanni\-an manifolds admitting the EMC with totally skew-symmetric torsion.
Our future research aims to extend these results by incorporating the \(Q\)-torsion condition instead of the \(A\)-torsion condition, which offers the potential for a richer geometric framework. We expect that analyzing the interaction involving the \(Q\)-torsion condition will lead to more general (specific) results and a deeper understanding of the geometry of weak structures on generalized Riemannian manifolds satisfying EMC with totally skew-symmetric torsion, including new classifications and possible applications in theoretical physics.
In our future analysis of the $Q$-torsion condition, we  examine the interaction between the metric $g$, the fundamental form $F$, and the self-adjoint tensor $Q$, paying special attention to the special cases, where the equalities $\nabla^g Q=0$, $\nabla Q=0$ and $[A,Q]=0$ are~satisfied.}

\bigskip

\noindent {\bf Acknowledgments.} M.Z. was partially supported by the Ministry of Education, Science and Technological Development of the Republic of Serbia, project no. 451-03-137/2025-03/200124 and by the Bulgarian Ministry of Education and Science, Scientific Programme "Enhancing the Research Capacity in Mathematical Sciences (PIKOM)", No. DO1-67/05.05.2022.

\end{document}